\documentclass[12pt,twoside]{amsart}
\usepackage{amssymb}
\usepackage{amscd}
\usepackage{xypic}

\title{On maximal Albanese dimensional varieties}
\author{Osamu Fujino} 
\date{2009/11/12, version 1.01}
\address{Department of Mathematics, Faculty of Science, 
Kyoto University, 
Kyoto 606-8502, Japan} 
\email{fujino@math.kyoto-u.ac.jp}
\newcommand{\Supp}[0]{{\operatorname{Supp}}}

\newcommand{\Exc}[0]{{\operatorname{Exc}}}

\newcommand{\Alb}[0]{{\operatorname{Alb}}}
\newcommand{\Var}[0]{{\operatorname{Var}}}

\newtheorem{thm}{Theorem}[section]
\newtheorem{lem}[thm]{Lemma}
\newtheorem{cor}[thm]{Corollary}
\newtheorem{prop}[thm]{Proposition}

\newtheorem{conj}[thm]{Conjecture}

\theoremstyle{definition}

\newtheorem{defn}[thm]{Definition}
\newtheorem{rem}[thm]{Remark}
\newtheorem*{ack}{Acknowledgments}      
\newtheorem*{notation}{Notation}

\begin{document}
\bibliographystyle{amsalpha+}
\begin{abstract}
We prove that every smooth projective variety with maximal Albanese dimension 
has a good minimal model. We also treat 
Ueno's problem on subvarieties of Abelian varieties.  
\end{abstract}
\maketitle

\tableofcontents

\section{Introduction}
In this short paper, we prove that 
every smooth projective 
variety with maximal Albanese dimension 
has a good minimal model. 
It is an easy consequence of \cite{bchm}. 
This paper is also a supplement to \cite{fujino1}. 
We will also treat Ueno's problem on subvarieties of 
Abelian varieties. 
A key point of this paper is the simple fact that 
there are no rational curves on an Abelian variety. 
The topics treated here seem to be easy exercises for experts. 

Let us recall the following minimal model conjecture.

\begin{conj}[Good minimal model conjecture (weak form)]\label{conj1}
Let $X$ be a smooth projective variety defined over $\mathbb C$.
Assume that $K_X$ is pseudo-effective.
Then there exists a normal projective variety $X'$ which
satisfies the following conditions{\em{:}}
\begin{itemize}
\item[(i)] $X'$ is birationally equivalent to $X$.
\item[(ii)] $X'$ has only $\mathbb Q$-factorial terminal singularities.
\item[(iii)] $K_{X'}$ is semi-ample.
\end{itemize}
In particular, the Kodaira dimension $\kappa (X)=\kappa (X, K_X)$ is non-negative.
We sometimes call $X'$ a {\em{good minimal model}} of $X$.
\end{conj}

\begin{rem}
By \cite{bchm},
we can replace (ii) with the following
slightly weaker condition:~(ii$'$) 
$X'$ has at most canonical singularities.
\end{rem}

The conjecture:~Conjecture \ref{conj1}
was established in the following cases.
\begin{itemize}
\item[(A)] $\dim X\leq 3$ (see, for example, \cite{fafa}).
\item[(B)] varieties of general type in 
any dimension (see, for example, \cite{bchm}).
\item[(C)] $\Delta$-regular divisors on complete toric varieties in 
any dimension (see \cite{ishii}).
\item[(D)] irregular fourfolds (see \cite{fujino4}).
\end{itemize}
As stated above, in this paper, 
we will prove that the conjecture:~Conjecture \ref{conj1}
holds for projective maximal Albanese dimensional varieties.
\begin{notation} 
For a proper birational 
morphism $f:X\to Y$, the 
{\em{exceptional locus}} 
$\Exc (f)\subset X$ is the locus 
where $f$ is not an isomorphism. 

We will freely use the basic notation and 
definitions in \cite{km} and \cite{ueno}. 
\end{notation}

\begin{ack}
The author 
was partially supported by The Inamori Foundation and by 
the Grant-in-Aid for Young Scientists (A) $\sharp$20684001 from 
JSPS. He thanks Professor Daisuke Matsushita whose 
questions made him consider this problem. 
\end{ack}

We will work over $\mathbb C$, the complex number field, throughout this 
paper. 

\section{Preliminaries} 

Let us recall the definition of maximal 
Albanese dimensional varieties. 

\begin{defn}
Let $X$ be a smooth projective variety. Let $\Alb(X)$ be 
the Albanese variety of $X$ and 
$\alpha: X\to \Alb (X)$ the corresponding Albanese 
map. 
We say that $X$ has {\em{maximal Albanese dimension}} 
if $\dim \alpha (X)=\dim X$. 
\end{defn}

\begin{rem}
A smooth projective variety $X$ has maximal Albanese dimension 
if and only if the cotangent bundle 
of $X$ is generically generated by 
its global sections, that is, 
$$
H^0(X, \Omega^1_X)\otimes \mathcal O_X\to 
\Omega^1_X
$$ 
is surjective at the generic point of $X$. It can 
be checked without any difficulties. 
\end{rem}

We note that the notion of maximal Albanese dimension 
is birationally invariant. 
So, we can define the notion of maximal Albanese dimension 
for singular varieties as follows. 

\begin{defn}
Let $X$ be a projective variety. 
We say that $X$ has {\em{maximal Albanese dimension}} 
if there is a resolution $\pi:\overline X\to X$ such 
that $\overline X$ has maximal Albanese dimension. 
\end{defn}

The following lemma is almost obvious by the definition of 
maximal Albanese dimensional varieties and 
the basic properties of 
Albanese mappings. 
We leave the proof for the reader's exercise. 

\begin{lem}\label{lem23} 
Let $X$ be a projective 
variety with maximal Albanese dimension. 
Let $\pi:\overline X\to X$ be a resolution and 
$\alpha:\overline X\to \Alb (\overline X)$ the Albanese 
mapping. 
Let $Y$ be a subvariety of $X$. 
Assume that 
$Y\not\subset \pi(\Exc (\pi)\cup \Exc (\beta))$, 
where $\beta:\overline X\to V$ is the Stein factorization of 
$\alpha:\overline X\to \Alb (\overline X)$. 
Then $Y$ has maximal Albanese dimension. 
\end{lem}

Let us recall some basic definitions. 

\begin{defn}[Iitaka's $D$-dimension and numerical 
$D$-dimension]
Let $X$ be a normal projective 
variety and $D$ a $\mathbb Q$-Cartier $\mathbb Q$-divisor. 
Assume that $mD$ is Cartier for a positive 
integer $m$. 
Let 
$$\Phi_{|tmD|}: X\dashrightarrow \mathbb P^{\dim |tmD|}
$$ 
be rational mappings given by 
linear systems $|tmD|$ for positive integers $t$. 
We define {\em{Iitaka's $D$-dimension}} 
\begin{eqnarray*}
\kappa (X, D)=\left\{
\begin{array}{ll} 
\underset{t>0}{\max} \dim \Phi _{|tmD|}(X), & {\text{if}} \ \ 
|tmD|\ne \emptyset 
\ \ {\text{for some}}\ \  t, 
\\ 
-\infty, & {\text{otherwise}}.
\end{array}\right. 
\end{eqnarray*}
In case $D$ is nef, 
we can also define the {\em{numerical 
$D$-dimension}} 
$$
\nu(X, D)=\max \{ \, e \,|\, D^e\not \equiv 0 \}, 
$$ 
where $\equiv$ denotes {\em{numerical equivalence}}. 
We note that $\nu(X, D)\geq \kappa (X, D)$ holds. 
\end{defn}

In this paper, we adopt the following definition of 
{\em{log terminal models}} for klt pairs. 

\begin{defn}[Log terminal models for klt pairs]\label{model}
Let $f:X\to S$ be a projective morphism of normal quasi-projective
varieties.
Suppose that $(X, B)$ is klt and let $\phi:X\dashrightarrow X'$ be  a
birational map of normal quasi-projective varieties over $S$, where
$X'$ is projective over $S$.
We put $B'=\phi_*B$.
In this case,
$(X', B')$ is a {\em{log terminal model}} of
$(X, B)$ over $S$ if the following conditions hold.
\begin{itemize}
\item[(i)] $\phi^{-1}$ contracts no divisors.
\item[(ii)] $(X', B')$ is a $\mathbb Q$-factorial klt pair.
\item[(iii)] $K_{X'}+B'$ is nef over $S$.
\item[(iv)] $a(E, X, B)<a(E, X', B')$ for all $\phi$-exceptional divisors
$E\subset X$.
\end{itemize}
\end{defn}
\section{Main results} 

\subsection{Minimal model}

First, we recall the following elementary but very 
important lemma. 

\begin{lem}[Negative rational curves]\label{lem32}
Let $X$ be a projective variety and $B$ an 
effective $\mathbb R$-divisor 
on $X$ such that 
$(X, B)$ is log canonical. 
Assume that $K_X+B$ is not nef. 
Then there exists a rational curve $C$ on $X$ such that 
$(K_X+B)\cdot C<0$. 
\end{lem}
\begin{proof} 
It is obvious by the cone theorem for log canonical pairs. 
See, for example, \cite[Proposition 3.21]{fujino3} and 
\cite[Section 18]{fufu}. 
We note that \cite{kawamata3} is sufficient when $(X, B)$ is 
klt. 
\end{proof}

By Lemma \ref{lem32}, we obtain the next lemma. 

\begin{lem}\label{lem322}
Let $f:X\to S$ be a proper surjective 
morphism between projective 
varieties. 
Let $B$ be an effective $\mathbb R$-divisor on $X$ 
such that $(X, B)$ is log canonical. 
Assume that 
$K_X+B$ is $f$-nef and 
$S$ contains no rational 
curves. 
Then $K_X+B$ is nef. 
\end{lem}
\begin{proof}
If $K_{X}+B$ is not nef, 
then there exists a rational 
curve $C$ on $X$ such that 
$(K_{X}+B)\cdot C<0$ by Lemma \ref{lem32}. 
Since $K_{X}+B$ is $f$-nef, $f(C)$ is not a point. 
On the other hand, $S$ contains no rational 
curves by the assumption. 
It is a contradiction. 
Therefore, $K_{X}+B$ is nef. 
\end{proof}

Therefore, the following 
lemma is obvious by Definition \ref{model} and 
Lemma \ref{lem322}. 

\begin{lem}\label{lem33} 
Let $f:X\to S$ be a proper 
surjective morphism between projective varieties. 
Let $B$ be an effective 
$\mathbb R$-divisor on $X$ such that 
$(X, B)$ is klt. 
Let $(X', B')$ be a log terminal 
model of $(X, B)$ over $S$. 
Assume that 
$S$ contains no rational curves. 
Then $(X', B')$ is a log terminal model 
of $(X, B)$.  
\end{lem}

We give an easy consequence of the main theorem 
of \cite{bchm}. 

\begin{thm}[Existence of log terminal models]\label{thm1} 
Let $X$ be a normal projective 
variety and 
$B$ an effective $\mathbb R$-divisor 
on $X$ such that $(X, B)$ is klt. 
Assume that $X$ has maximal Albanese 
dimension. 
Then $(X, B)$ has a log terminal model.  
\end{thm}
\begin{proof}
Let $\pi:\overline X\to X$ be a resolution 
and $\alpha:\overline X\to \Alb (\overline X)$ the Albanese mapping 
of $\overline X$. 
Since $X$ has only rational singularities, $\overline X\to S=\alpha (\overline 
X)$ decomposes as 
$$
\alpha:\overline X\overset{\pi}\longrightarrow 
X\overset{f}\longrightarrow S. 
$$ 
See, for example, \cite[Lemma 2.4.1]{bs}. 
Since $X$ has maximal Albanese dimension, 
$f:X\to S$ is generically finite. 
By \cite[Theorem 1.2]{bchm}, 
there exists a log terminal model 
$f:(X', B')\to S$ of 
$(X, B)$ over $S$. 
By Lemma \ref{lem33}, $(X', B')$ is a log terminal 
model of $(X, B)$. 
\end{proof}

The following corollary is obvious by Theorem \ref{thm1}. 

\begin{cor}[Minimal models]\label{cor12}
Let $X$ be a smooth projective variety with maximal 
Albanese dimension. 
Then $X$ has a minimal model. 
\end{cor}

\subsection{Abundance theorem}
Let us consider the abundance theorem for 
maximal Albanese dimensional varieties. 

\begin{thm}[Abundance theorem]\label{thm13}
Let $X$ be a projective variety with only 
canonical singularities. Assume that 
$X$ has maximal Albanese dimension. 
If $K_X$ is nef, then $K_X$ is semi-ample. 
\end{thm}
\begin{proof}
By Lemma \ref{lem35}, we have 
$\kappa (X, K_X)\geq 0$. 
Therefore, if $\nu(X, K_X)=0$, then 
$\kappa (X, K_X)=\nu(X, K_X)=0$. 
Thus, $X$ is an Abelian variety by Proposition \ref{prop36}. 
In particular, $K_X$ is semi-ample. 
By Lemma \ref{lem14}, it is sufficient 
to see that $\nu(X, K_X)>0$ implies 
$\kappa (X, K_X)>0$ because $\nu(X, K_X)=\kappa (X, K_X)$ 
means $K_X$ is semi-ample by \cite[Theorem 1.1]{kawamata2} and 
\cite[Corollary 2.5]{fujino2}. 
See also Remark \ref{re-new} below.  
By Proposition \ref{prop36}, 
we know that $\nu(X, K_X)>0$ implies $\kappa (X, K_X)>0$. 
We finish the proof.    
\end{proof}

\begin{lem}\label{lem35}
Let $X$ be a projective variety with only canonical singularities. 
Assume that 
$X$ has maximal Albanese dimension. 
Then $\kappa (X, K_X)\geq 0$. 
\end{lem}
\begin{proof}
Since $X$ has only canonical singularities, we can 
assume that $X$ is smooth by replacing $X$ with 
its resolution. 
Then this lemma is obvious by the basic properties of 
the Kodaira dimension (cf.~\cite[Theorem 6.10]{ueno}). 
Note that every subvariety of an 
Abelian variety has non-negative Kodaira dimension 
(cf.~\cite[Lemma 10.1]{ueno}). 
\end{proof}

\begin{prop}\label{prop36} 
Let $X$ be a projective variety with only 
canonical singularities. 
Assume that 
$X$ has maximal Albanese dimension, 
$K_X$ is nef, and $\kappa (X, K_X)=0$. 
Then $X$ is an Abelian variety. In particular, 
$X$ is smooth and 
$K_X\sim 0$. 
\end{prop}
\begin{proof}
We know that 
$f:X\to S=\Alb (\overline X)$ is birational 
by \cite[Theorem 1]{kawamata1}, where 
$\overline X$ is a resolution of $X$ and 
$\Alb (\overline X)$ 
is the Albanese variety of $\overline X$. 
We can write 
$$
K_X=f^*K_S+E
$$ 
such that $E$ is effective and $\Supp E=\Exc (f)$. 
Since $K_X$ is nef and $K_S\sim 0$, we obtain $E=0$. 
This means that $f:X\to S$ is an isomorphism. 
\end{proof}

\begin{lem}[{cf.~\cite[Theorem 7.3]{kawamata2}}]\label{lem14}
Let $X$ be a projective variety with only canonical singularities. 
Assume that $X$ has maximal Albanese dimension and 
that $K_X$ is nef. 
If $\nu(X, K_X)>0$ implies $\kappa (X, K_X)>0$, 
then $\nu(X, K_X)=\kappa (X, K_X)$. 
\end{lem}
\begin{proof}[Sketch of the proof]
The proof of \cite[Theorem 7.3]{kawamata2} works without any changes. 
We give some comments for the reader's convenience. 
We use the same notation as in the proof of \cite[Theorem 7.3]{kawamata2}. 
By Lemma \ref{lem23}, $W$ has maximal Albanese dimension. 
On the other hand, $\kappa (W)=0$ by the 
construction. 
By Proposition \ref{prop36}, 
$W_{min}$ is an Abelian variety. In particular, 
$K_{W_{min}}\sim 0$. 
\end{proof}

\begin{rem}
In the above proof of the abundance theorem, we did not use \cite[Theorem 8.2]{kawamata2.5}. It is because we have 
Proposition \ref{prop36}. 
\end{rem}

\begin{rem}\label{re-new}
On the assumption that 
$\nu(X, K_X)=\kappa (X, K_X)$, 
$K_X$ is semi-ample if and only if 
the canonical ring of $X$ is finitely generated. 
Therefore, by \cite{bchm}, 
we can check the semi-ampleness of 
$K_X$ without appealing \cite[Theorem 1.1]{kawamata2} 
and \cite[Corollary 2.5]{fujino2}. 
\end{rem}

By Corollary \ref{cor12} and Theorem \ref{thm13}, 
we obtain the main result of this short paper. 
 
\begin{cor}[Good minimal models]\label{cor25}
Let $X$ be a smooth projective variety with maximal Albanese dimension. 
Then $X$ has a good minimal model. 
This means that 
there is a normal projective variety $X'$ with only $\mathbb Q$-factorial 
terminal singularities such that 
$X'$ is birationally equivalent to $X$ and 
$K_{X'}$ is semi-ample. 
\end{cor} 

\subsection{Iitaka--Viehweg's conjecture}
By combining Corollary \ref{cor25} with the main theorem 
of \cite{kawamata2.5}, we obtain the following theorem. 
We write it for the reader's convenience. 
For the details, see \cite{kawamata2.5}. 

\begin{thm}[{cf.~\cite[Theorem 1.1]{kawamata2.5}}]\label{thm36} 
Let $f:X\to S$ be a surjective morphism between 
smooth projective varieties with connected fibers 
and $\mathcal L$ a line bundle on $S$. 
Assume that the geometric generic fiber of $f$ 
has maximal Albanese dimension. 
Then the following assertions hold{\em{:}} 
\begin{itemize}
\item[(i)] There exists a positive integer $n$ such that 
$$
\kappa (S, \widehat {\det}(f_*(\omega^n_{X/S})))\geq \Var (f). 
$$
\item[(ii)] If $\kappa (S, \mathcal L)\geq 0$, then 
$$
\kappa (X, \omega_{X/S}\otimes f^*\mathcal L)\geq 
\kappa (X_\eta)+\max \{\kappa (S, \mathcal L), \Var (f)\}, 
$$ 
where $X_\eta$ is the generic fiber of $f$. 
\end{itemize}
\end{thm}
The next corollary is a special case of \cite[Corollary 1.2]{kawamata2.5} 
(cf.~\cite[Theorem 0.2]{fujino1}. 

\begin{cor}\label{cor37}
Under the same assumptions as in {\em{Theorem \ref{thm36}}}, 
\begin{itemize}
\item[(i)] $\kappa (X, \omega_{X/S})\geq \kappa (X_\eta)
+\Var (f)$, and 
\item[(ii)] if $\kappa (S)\geq 0$, then 
$\kappa (X)\geq \kappa (X_\eta)+\max \{\kappa (S), \Var (f)\}$. 
\end{itemize}
\end{cor}

\subsection{Ueno's problem}
The final theorem is a supplement to 
\cite[Remark 10.13]{ueno}. It is an easy consequence of 
Lemma \ref{lem32}. 

\begin{thm}\label{thm38}
Let $X$ be a projective variety and $B$ an effective 
$\mathbb R$-divisor on $X$ such that 
$(X, B)$ is log canonical. 
Assume that there are no rational curves on $X$. 
Then $K_X+B$ is nef. 

Furthermore, we assume that 
$(X, B)$ is klt, $B$ is an effective $\mathbb Q$-divisor, and 
$K_X+B$ is big. 
Then $K_X+B$ is ample. 
\end{thm}
\begin{proof}
The first half of this theorem is obvious by Lemma \ref{lem32}. 
If $(X, B)$ is klt, 
$B$ is an effective $\mathbb Q$-divisor, 
and $K_X+B$ is big, 
then $K_X+B$ is semi-ample by the 
base point free theorem since $K_X+B$ is nef. 
Then there exists a birational 
morphism $$f=\Phi _{|m(K_X+B)|}: X\to S$$ 
with $f_*\mathcal O_X\simeq \mathcal O_S$ for some 
large and divisible 
integer $m$. 
By the construction, there is an 
ample Cartier divisor $H$ on $S$ such 
that $m(K_X+B)\sim f^*H$.  
Assume that $f$ is not an isomorphism. 
Let $A$ be an $f$-ample Cartier divisor. 
Then there is an effective Cartier divisor $D$ on $X$ such that 
$D\sim -A+f^*lH$ for some large integer $l$. 
Therefore, $(X, B+\varepsilon D)$ is klt and 
$K_X+B+\varepsilon D$ is not nef 
for $0<\varepsilon \ll 1$ since $f$ is not an isomorphism. 
By Lemma \ref{lem32}, there exists 
a rational curve $C$ on $X$ such that 
$(K_X+B+\varepsilon D)\cdot C<0$. 
It is a contradiction because there are no rational curves on $X$. 
Therefore, $f$ is an isomorphism. 
Thus, $K_X+B$ is ample. 
\end{proof}

\begin{cor}[{cf.~\cite[Remark 10.13]{ueno}}]
Let $W$ be a submanifold of a complex torus $T$ with 
$\kappa (W)=\dim W$. 
Then $K_W$ is ample.  
\end{cor}
\begin{proof} 
By \cite[Lemma 10.8]{ueno}, there is an Abelian 
variety $A$ which is a complex subtorus of $T$ such that 
$W\subset A$. 
Thus, $W$ is projective. 
Then, by Theorem \ref{thm38}, we obtain $K_W$ is ample. 
\end{proof}

\ifx\undefined\bysame
\newcommand{\bysame|{leavemode\hbox to3em{\hrulefill}\,}
\fi

\end{document}